\documentclass[12pt]{amsart}
\usepackage{amsmath, amsthm, amscd, amsfonts}
\usepackage{color}

\newtheorem{theorem}{Theorem}[section]
\newtheorem{lemma}[theorem]{Lemma}

\newtheorem{corollary}[theorem]{Corollary}
\theoremstyle{definition}
\newtheorem{definition}[theorem]{Definition}

\newtheorem{construction}[theorem]{Construction}
\theoremstyle{remark}

\numberwithin{equation}{section}

\newfont{\kh}{msbm10}

\begin{document}
\title[Induced Representations of Hilbert Modules]
{Induced Representations of Hilbert Modules over Locally C*-algebras
and the imprimitivity theorem}
\author{Kh. Karimi}
\address{Khadijeh Karimi,
\newline Department of Mathematics, Shahrood University, P.
O. Box 3619995161-316, Shahrood, Iran.}
\email{kh\underline{\space}karimi5005@yahoo.com}
\author{K. Sharifi}
\address{Kamran Sharifi,
\newline Current Address: Mathematisches Institut,
Fachbereich Mathematik und Informatik der Universit\"{a}t M\"{u}nster,
Einsteinstrasse 62, 48149 M\"{u}nster, Germany.
\newline Permanent Address: Department of Mathematics,
Shahrood University, P.
O. Box 3619995161-316, Shahrood, Iran.}
\email{sharifi.kamran@gmail.com}


\subjclass[2010]{Primary 46L08; Secondary 46K10, 46L05}
\keywords{Hilbert modules, locally C*-algebras, module
maps,  Morita equivalence, induced representations}
\begin{abstract}

We study induced representations of Hilbert
modules over locally C*-algebras and their non-degeneracy.
We show that if $V$ and $W$ are Morita equivalent Hilbert modules over locally C*-algebras
$A$ and $B$, respectively, then there exists a bijective correspondence between
equivalence classes of non-degenerate representations of $V$ and $W$.
\end{abstract}
\maketitle


\section{Introduction}
Morita equivalence and induced representations of C*-algebras were first introduced
by Rieffel \cite{rif1,rif2}.  Two C*-algebras $A$ and $B$ are
Morita equivalent if there exists a full Hilbert $A$-module $E$ such that
$B$ is isomorphic to the C*-algebra $K_{A}(E)$ of all compact operators on $E$.
Some properties of C*-algebras that are preserved under Morita equivalence
were investigated in ~\cite{Ara,bee,rae,zet}. Indeed,
Rieffel defined induced representations of C*-algebras, that are now known as
Rieffel induced representations,
by using tensor products of Hilbert modules and established an equivalence between the categories
of non-degenerate representations of Morita equivalent C*-algebras.
Joita \cite{mar5,mar4}
defined the notions of Morita equivalence and induced representations
in the category of locally C*-algebras. Joita and Moslehian \cite{ms}
have recently introduced a notion of Morita equivalence in
the category of Hilbert C*-modules considered to obtain induced representations of Hilbert
modules over locally C*-algebras. This enables us to prove the imprimitivity theorem for
induced representations of Hilbert modules over locally C*-algebras.

Let us quickly recall the
definition of locally C*-algebras and Hilbert modules over them.
A locally C*-algebra is a complete Hausdorff complex
topological $*$-algebra $A$ whose
topology is determined by its continuous C*-seminorms in the
sense that the net $\{a_{i}\}_{i \in I}$ converges to $0$ if and only
if the net $\{p(a_i)\}_{i \in I}$ converges to $0$ for every
continuous C*-seminorm $p$ on $A$.
Such algebras appear in the study
of certain aspects of C*-algebras such as tangent algebras of C*-algebras, a domain of
closed $*$-derivations on C*-algebras, multipliers of
Pedersen's ideal, noncommutative analogues of classical Lie groups, and K-theory.
These algebras were first introduced by Inoue \cite{ino} as a generalization of
C*-algebras and studied more in
\cite{fra, phil1} with different names.
A (right) {\it pre-Hilbert module} over a locally C*-algebra
$A$ is a right $A$-module $E$
compatible with the complex algebra structure and equipped with an
$A$-valued inner product $\langle \cdot , \cdot \rangle
: E \times E \to A \,, \ (x,y) \mapsto \langle x,y
\rangle$, which is $A$-linear in the second variable $y$
and has the properties:
\[\langle x,y \rangle=\langle y,x \rangle ^{*}, \ {\rm and} \
 \langle x,x \rangle \geq 0 \ \ {\rm with} \
   {\rm equality} \ {\rm if} \ {\rm and} \ {\rm only} \
   {\rm if} \ x=0.\]

A pre-Hilbert $A$-module $E$ is a Hilbert
$A$-module if $E$ is complete with respect to the
topology determined by the family of seminorms $ \{
\overline{p}_E \}_{p \in S(A)}$,  where $\overline{p}_E
( \xi) = \sqrt{ p( \langle \xi, \xi \rangle)}$, $ \xi \in E$.
Hilbert modules over locally C*-algebras have
been studied systematically in the book \cite{mar1} and the
papers \cite{mar00, phil1, SHAMal}.

Joita and Moslehian \cite{ms}, and Skeide \cite{ske} defined
Morita equivalence for Hilbert
C*-modules with two different methods. In the recent sense of Joita and Moslehian,
two Hilbert modules $V$ and $W$ over C*-algebras $A$ and $B$, respectively,
are called Morita equivalent if $K_{A}(V)$ and $K_{B}(W)$ are strong
Morita equivalent as C*-algebras. We consider this definition,
which is weaker than Skeide's definition and also fitted to our paper.

In this paper, we first present some definitions and basic facts about locally
C*-algebras and Hilbert modules over them.  In~\cite{ske1}, Skeide proved
that if $E$ is a Hilbert module over a C*-algebra $A$, then every representation
of $A$ induces a representation of $E$. We use this fact to reformulate
the induced representations of Hilbert C*-modules and some of their properties
which have been
studied in \cite{Abbaspour}. These enable us to obtain the notion of
induced representations of Hilbert modules over locally C*-algebras.
We finally define the concept of Morita equivalence
for Hilbert modules over locally C*-algebras.
We prove that two full Hilbert modules over locally C*-algebras are Morita
equivalent if and only if their underlying locally C*-algebras
are strong Morita equivalent and then we give a module version of
the imprimitivity theorem. Indeed, we show
that for Morita equivalent
Hilbert modules $V$ and $W$ over locally C*-algebras
$A$ and $B$, respectively, there is a bijective correspondence between
equivalence classes of non-degenerate representations of $V$ and $W$.


\section{Preliminaries}
Let $A$ be a locally C*-algebra, $S(A)$  the set of all
continuous C*-seminorms on $A$ and $p\in S(A)$. We set
$N_{p}=\{a\in A:\ p(a)=0\}$, then $A_p=A/N_{p}$ is a C*-algebra in the norm induced
by $p$. For $p,q\in S(A)$ with $p\geq q$, the surjective
morphisms $\pi_{pq}: A_p \to A_q$ defined by $\pi_{pq}(a+N_{p})=a+N_{q}$
induce the inverse system $\{A_p; \pi_{pq}\}_{p,q\in S(A), \, p\geq q}$
of C*-algebras and $A = \varprojlim_{p}A_p$, i.e.,
the locally C*-algebra $A$ can be identified with $ \varprojlim_{p}A_p$.
The canonical map from $A$ onto $A_p$ is
denoted by $\pi_p$ and $a_{p}$ is reserved to denote $a+N_{p}$.
A morphism of locally C*-algebras is a continuous morphism of $*$-algebras.
An isomorphism of locally C*-algebras is a morphism of locally C*-algebras
which possesses an inverse morphism of locally C*-algebras.

A representation of a locally C*-algebra $A$ is a continuous $*$-morphism $\varphi: A \to B(H)$,
where $B(H)$ is the C*-algebra of all bounded linear maps on a Hilbert space $H$.
If $(\varphi,H)$ is a representation of $A$, then there is $p\in S(A)$ such that
$\|\varphi(a)\|\leq p(a)$, for all $a\in A$.
The representation $(\varphi_{p},H)$ of $A_{p}$, where $\varphi_{p} \circ  \pi_{p}=\varphi$
is called a representation of $A_{p}$ associated to $(\varphi,H)$.
We refer to \cite{fra, mar4} for basic facts and definitions about the
representation of locally C*-algebras.

Suppose $E$ is a Hilbert $A$-module and $\langle E,E\rangle$ is the
closure of linear span of $\{ \langle x,y\rangle: ~ x,y\in E \}$.
The Hilbert $A$-module $E$ is called {\it full} if $\langle E,E\rangle=A$.
One can always consider any Hilbert $A$-module as a full Hilbert module
over locally C*-algebra $\langle E,E\rangle$.
For each $p\in S(A), N_{p}^{E}=\{\xi\in E: \ \bar{p}_{E}(\xi)=0\}$
is a closed submodule of $E$ and $E_{p}=E/N_{p}^{E}$ is a
Hilbert $A_{p}$-module with the action $(\xi+N_{p}^{E})\pi_{p}(a)=\xi a+N_{p}^{E}$
and the inner product $\langle\xi+N_{p}^{E},\eta+N_{p}^{E}\rangle=\pi_{p}(\langle\xi,\eta\rangle).$
The canonical map from $E$ onto $E_{p}$ is denoted by $\sigma_{p}^{E}$ and
$\xi_{p}$ is reserved to denote $\sigma_{p}^{E}(\xi).$
For $p,q\in S(A)$ with $p\geq q$, the surjective morphisms
$\sigma_{pq}^{E}:E_{p} \to E_q$ defined by
$\sigma_{pq}^{E}(\sigma_{p}^{E}(\xi))=\sigma_{q}^{E}(\xi)$
induce the inverse system $\{E_{p};\  A_p;\  \sigma_{pq}^{E},\ \pi_{pq} \}_{p,q\in S(A), \, p\geq q}$
of Hilbert C*-modules in the following sense:
\begin{itemize}
\item $\sigma_{pq}^{E}(\xi_p a_p)= \sigma_{pq}^{E}(\xi_p)\pi_{pq}(a_p),\
 \xi_p\in E_p,\  a_p\in A_p,\  p,q\in S(A),\  p\geq q,$
\item $\langle\sigma_{pq}^{E}(\xi_p),\sigma_{pq}^{E}(\eta_p)\rangle=\pi_{pq}
(\langle\xi_p,\eta_p\rangle),\  \xi_p,\eta_p\in E_p,\  p,q\in S(A),\  p\geq q,$
\item $\sigma_{qr}^{E}\,  \circ \, \sigma_{pq}^{E}=\sigma_{pr}^{E}$\ \   if\ \
  $p,q,r\in S(A)$ and $p\geq q\geq r,$
\item $\sigma_{pp}^{E}(\xi_p)=\xi_p,\ \xi\in E,\ p\in S(A).$
\end{itemize}
In this case, $\varprojlim_{p}E_{p}$ is a
Hilbert $A$-module which can be identified with $E$.
Let $E$ and $F$ be Hilbert $A$-modules and $T:E\to F$ an $A$-module map.
The module map $T$ is called bounded if for each $p\in S(A)$
there is $k_p>0$ such that $\bar{p}_{F}(Tx)\leq k_{p}\ \bar{p}_{E}(x)$ for all $x\in E.$
The module map $T$ is called adjointable if there exists
an $A$-module map $T^{*}:F\to E$ with the property $\langle Tx,y\rangle=\langle x,T^{*}y\rangle$
for all $x\in E, y\in F.$ It is well-known that every adjointable map is bounded.
The set $L_{A}(E,F)$ of all bounded adjointable $A$-module maps from $E$
into $F$ becomes a locally convex space with the topology defined by the
family of seminorms $\{\tilde{p}\}_{p\in S(A)}$, where $\tilde{p}(T)=\|(\pi_{p})_{*}(T)\|_{ L_{A_p}(E_p,F_p)}$
and $(\pi_{p})_{*}:L_{A}(E,F)\to L_{A_p}(E_p,F_p)$ is defined by
$(\pi_p)_{*}(T)(\xi+N_{p}^{E})=T\xi+N_{p}^{F}$ for
all $T\in L_{A}(E,F), ~ \xi\in E$. For $p,q\in S(A)$ with $p\geq q$,
the morphisms $(\pi_{pq})_{*}:L_{A_p}(E_p,F_p)\to L_{A_q}(E_q,F_q)$
defined by $(\pi_{pq})_{*}(T_{p})(\sigma_{q}^{E}(\xi))=\sigma_{pq}^{F}(T_p(\sigma_{p}^{E}(\xi)))$ induce
the inverse system \[\{ L_{A_p}(E_p,F_p); (\pi_{pq})_{*}\}_{p,q\in S(A), \, p\geq q}\]
of Banach spaces such that $\varprojlim_{p} L_{A_p}(E_p,F_p)$ can be identified to $L_{A}(E,F)$.
In particular, topologizing, $L_{A}(E,E)$ becomes a locally
C*-algebra which is abbreviated by $L_{A}(E)$.
The set of all compact operators $K_{A}(E)$ on $E$ is defined as the closed
linear subspace of $L_A(E)$ spanned by
$\{\theta_{x,y}: \ \theta_{x,y}(\xi)=x\langle y,\xi\rangle ~ {\rm for~ all}~ x,y,\xi \in E\}$.
 This is a locally C*-subalgebra and a two-sided ideal of $L_{A}(E)$; moreover,
$K_{A}(E)$ can be identified to $\varprojlim_{p}\ K_{A_p}(E_p)$.

Let $V$ and $W$ be Hilbert modules over locally C*-algebras $A$ and $B$, respectively,
and $\Psi:A\to L_{B}(W)$ a continuous $*$-morphism. We can regard $W$ as a
left $A$-module by $(a,y)\to \Psi(a)y, ~a\in A, ~y\in W$. The right $B$-module
$V\otimes_{A}W$ is a pre-Hilbert module
with the inner product given by
$\langle x\otimes y, z\otimes t\rangle=\langle y,\Psi(\langle x,z\rangle)t\rangle$.
We denote by $V\otimes_{\Psi}W$ the completion of $V\otimes_{A}W$,
cf. \cite{mar3} for more detailed information.

\section{Induced representations of Hilbert modules}
In this section, we first study induced representations of Hilbert C*-modules and then
we reformulate them in the context of Hilbert modules over locally C*-algebras.

Let $H$ and $K$ be Hilbert spaces. Then the space $B(H,K)$ of all bounded operators from $H$ into $K$
can be considered as a Hilbert $B(H)$-module with the module action
$(T,S)\to TS$, $T\in B(H,K)$ and $S\in B(H)$ and the inner product defined by
$\langle T,S\rangle=T^{*}S$, $T,S\in B(H,K)$. Murphy \cite{mur} showed that any
Hilbert C*-module can be represented as a submodule of the concrete Hilbert module $B(H,K)$
for some Hilbert spaces $H$ and $K$.
This allows us to extend the notion of a representation from the
context of C*-algebras to the context of Hilbert C*-modules.
Let $V$ and $W$ be two Hilbert modules over C*-algebras $A$ and $B$, respectively,
and $\varphi:A \to B$ be a morphism of C*-algebras. A map $\Phi: V \to W$ is said to a
$\varphi$-morphism if $\langle\Phi(x),\Phi(y)\rangle=\varphi(\langle x,y\rangle)$
for all $x,y\in V$. A $\varphi$-morphism $\Phi:V\to B(H,K)$,
where $\varphi:A \to B(H)$ is a representation of $A$ is called a representation of $V$.
When $\Phi$ is a representation of $V$, we assume that an associated
representation of $A$ is denoted by
the same lowercase letter $\varphi$, so we will not explicitly mention $\varphi$.
Let $\Phi:V \to B(H,K)$ be a representation of a Hilbert $A$-module $V$.
We say $\Phi$ is a non-degenerate representation if
$\overline{\Phi(V)(H)}=K$ and $\overline{\Phi(V)^{*}(K)}=H$.
Two representations $\Phi_{i}:V\to B(H_{i},K_{i})$ of $V$, $i=1,2$
are said to be unitarily equivalent if there are
unitary operators $U_{1}:H_{1}\to H_{2}$ and $U_{2}:K_{1}\to K_{2}$,
such that $U_{2}\Phi_{1}(v)=\Phi_{2}(v)U_{1}$ for all $v\in V$.
Representations of Hilbert modules have been investigated in ~\cite{Abbaspour,aram,ske1}.

\begin{lemma} \label{Rep1}
Let $V$ be a full Hilbert $A$-module and
$\Phi_{1}:V\to B(H_{1},K_{1})$ and $\Phi_{2}:V\to B(H_{2},K_{2})$
two non-degenerate representations of $V$. If $\Phi_{1}$
and $\Phi_{2}$ are unitarily equivalent, then $\varphi_{1}$
and $\varphi_{2}$ are unitarily equivalent.
\end{lemma}
\begin{proof}
Let $U_1:H_1 \to H_2$ and $U_2:K_1\to K_2$ be unitary operators and
$U_2\Phi_1(x)=\Phi_2(x)U_1$ for all $x\in V$. Then we have
\[U_{1}\varphi_1(\langle x,y \rangle)h=U_1\Phi_{1}(x)^{*}\Phi_{1}(y)h=
\Phi_{2}(x)^{*}\Phi_{2}(y)U_1h=\varphi_2(\langle x,y \rangle)U_{1}h,\]
for every $x,y\in V$ and $h\in H_1$.
Since $V$ is full, we conclude that $U_{1}\varphi_1(a)h=\varphi_2(a)U_{1}h$ for every
$a \in A$ and $h\in H_1$, and consequently, $\varphi_1$ and $\varphi_2$ are unitarily equivalent.
\end{proof}
Skeide ~\cite{ske1} recovered the result of
Murphy by embedding  every Hilbert $A$-module
$E$ into a matrix C*-algebra as a lower submodule. He proved that every
representation of $B$ induces a representation of $E$. We describe his
induced representation as follows.
\begin{construction}\label{const1}
Let $B$ be a C*-algebra and $E$ a Hilbert $B$-module
and $\varphi: B \to B(H)$ a $*$-representation
of $B$. Define a sesquilinear form $\langle.,.\rangle$ on the vector space $E\otimes_{alg}H$ by
$\langle x\otimes h,y\otimes k\rangle=\langle h,\varphi(\langle x,y\rangle)k\rangle_{H},$
where $\langle.,.\rangle_{H}$ denotes the inner product on the Hilbert space $H$.
By \cite[Proposition 3.8]{ske1}, the sesquilinear form is positive and so $E\otimes_{alg}H$
is a semi-Hilbert space. Then $(E\otimes_{alg}H)/N_{\varphi}$
is a pre-Hilbert space with the inner product defined by
\[\langle x\otimes h+N_{\varphi} \ ,\  y\otimes k+N_{\varphi}\rangle=\langle x\otimes h,y\otimes k\rangle,\]
where $N_{\varphi}$ is the vector subspace of $E\otimes_{alg}H$
generated by $\{x\otimes h\in E\otimes_{alg}H: \langle x\otimes h,x\otimes h\rangle=0 \}$.
The completion of $(E\otimes_{alg}H)/N_{\varphi}$ with respect to the above inner product
is denoted by $_{E}H$.
We identify the elements $x\otimes h$ with the equivalence classes $x\otimes h+N_{\varphi}\in$  $_{E}H$.
Suppose $x\in E$ and $L_{x}h=x\otimes h$ then
$\| L_{x}h\|^2= \langle h,\varphi(\langle x,x\rangle)h\rangle\leq \| h \|^2\|x\|^2$, i.e.
$L_{x}\in B(H,_{E}H)$.
We define $\eta_{\varphi}: E \to B(H,_{E}H)$ by $\eta_{\varphi}(x)=L_{x}$. Then
for $x,x'\in E$,  $h,h'\in H$ and $b\in B$ we have
$\langle \eta_{\varphi}(x),\eta_{\varphi}(x')\rangle=\varphi(\langle x,x'\rangle)$ and
$\eta_{\varphi}(xb)=\eta_{\varphi}(x)\varphi(b)$, and so
$ \eta_{\varphi}$ is a representation of $E$.
\end{construction}
\begin{lemma}\label{Rep2}
Let $\varphi_1:B \to B(H_1)$ and $\varphi_2:B \to B(H_2)$ be two
non-degenerate representations of $B$. If $\varphi_1$ and $\varphi_2$
are unitarily equivalent, then $\eta_{\varphi_{1}}$ and $\eta_{\varphi_{2}}$
are unitarily equivalent.
\end{lemma}
\begin{proof}
Suppose $U:H_1\to H_2$ is a unitary operator such that $U\varphi_1(b)=\varphi_2(b)U$ for all $b\in B$.
Then $id_{E}\otimes U: E\otimes_{alg} H_1 \to E \otimes_{alg} H_2$ given
by $x\otimes h_1 \mapsto x\otimes h_2$ can be extended to a unitary operator $V$ from
$_{E}H_1$ onto $_{E}H_2$ and $V\eta_{\varphi_{1}}(x)=\eta_{\varphi_2}(x)U$ for all $x\in E$.
Hence, $\eta_{\varphi_{1}}$ and $\eta_{\varphi_{2}}$ are unitarily equivalent.
\end{proof}
The above argument enables us to extend the Rieffel induced representations from
the case of C*-algebras to the context of Hilbert C*-modules.
For this, let $V$ and $W$ be two full Hilbert modules
over C*-algebras $A$ and $B$, respectively.
Let $E$ be a Hilbert $B$-module and $A$ acts as adjointable
operators on the Hilbert C*-module $E$, and $\Phi: W\to B(H,K)$
is a non-degenerate representation of $W$. Using \cite[Proposition 2.66]{rae}, the formula
$_{E}^{A}\varphi(x\otimes h)=(a.x)\otimes h$
extends to obtain a (Rieffel induced) representation of $A$
as bounded operators on Hilbert space $_{E}H$.
In view of Construction \ref{const1}, the representation $_{E}^{A}\varphi:A \to B(_{E}H)$
of the C*-algebra $A$ obtains the representation
$\eta_{_{E}^{A}\varphi}:V \to B(_{E}H,\  _{V}(_{E}H))$ of the Hilbert $A$-module $V$.
The constructed representation $\eta_{_{E}^{A}\varphi}$ is called
the {\it Rieffel induced representation}
from $W$ to $V$ via $E$ and denoted by $_{E}^{V}\Phi$.
The following result can be found in \cite[Proposition 3.3]{Abbaspour}
that we derive  from Lemmas \ref{Rep1} and \ref{Rep2}.
Our argument seems to be shorter.

\begin{lemma}\label{Rep3}
Let $W$ be a full Hilbert $B$-module and $\Phi_{1}:W\to B(H_{1},K_{1})$
and $\Phi_{2}:W\to B(H_{2},K_{2})$ two non-degenerate representations of $W$.
If $\Phi_{1}$ and $\Phi_{2}$ are unitarily equivalent,
then $_{E}^{V}\Phi_{1}$ and $_{E}^{V}\Phi_{2}$ are unitarily equivalent.
\end{lemma}

\begin{corollary}\label{Rep4}
If $\Phi:W\to B(H,K)$  and
$\oplus_{i \in I} \Phi_{i}: W \to B(\oplus_{i \in I} H_{i},\oplus_{i \in I} K_{i})$
are unitarily equivalent,
then $_{E}^{V} \Phi$ and $\oplus_{i \in I} \ _{E}^{V} \Phi_{i}$ are unitary equivalent.
\end{corollary}

Now, we reformulate representations of the Hilbert module
from the case of C*-algebras to the case of locally C*-algebras.
Let $V$ and $W$ be two Hilbert modules over locally C*-algebras $A$ and $B$,
respectively, and $\varphi:A \to B$ a morphism of locally C*-algebras.
A map $\Phi: V \to W$ is said to be a $\varphi$-morphism if
$\langle \Phi(x), \Phi(y) \rangle=\varphi(\langle x,y\rangle)$, for all $x,y\in V$.
A $\varphi$-morphism $\Phi:V\to B(H,K)$, where $\varphi:A \to B(H)$ is a
representation of $A$, is called a representation of $V$.
We can define non-degenerate representations and unitarily
equivalent representations for Hilbert modules over
locally C*-algebras like a Hilbert C*-modules case.

Suppose $A$ is a locally C*-algebra, $V$ is a Hilbert $A$-module
and $\varphi:A \to B(H)$ is a representation of $A$ on some Hilbert space $H$.
Suppose $p \in S(A)$ and $\varphi_{p}$ is a representation of $A_{p}$ associated to $\varphi$; then
there exist a Hilbert space $K$ and a representation $\Phi_{p}:V_{p}\to B(H,K)$
which is a $\varphi_{p}$-morphism. For details we refer to the proof of \cite[Theorem 3.1]{mur}.
It is easy to see that the map
$\Phi:V \to B(H,K)$, $\Phi(v)=\Phi_{p}(\sigma_{p}^{V}(v))$ is a
$\varphi$-morphism, i.e., it is a representation of $V$.

\begin{lemma}\label{Rep5}
Let $V$ be a Hilbert module over locally C*-algebra $A$ and $\Phi:V \to B(H,K)$
 a representation of $V$. If $p \in S(A)$ and $\varphi_{p}$ is a representation of $A_{p}$
associated to $\varphi$, then the map $\Phi_{p}:V_{p}\to B(H,K)$,
$\Phi_{p}(\sigma_{p}^{V}(v)) = \Phi(v)$ is a $\varphi_{p}$-morphism.
Specifically, $\Phi_{p}$ is a representation of $V_{p}$ and $\Phi$
is non-degenerate if and only if $\Phi_{p}$ is. In this case,
we say that $\Phi_{p}$ is a representation of $V_{p}$ associated to $\Phi$.
\end{lemma}
\begin{proof} Let $v,v' \in V$ and $ \overline{p}_{V}(v-v')=0$.
Since $\| \varphi(a) \| \leq p(a)$ for all $a \in A$, we have
$\langle \Phi(v-v'),\Phi(v-v')\rangle=\varphi(\langle v-v',v-v' \rangle)=0$,
which shows $\Phi_{p}$ is well-defined. We also have
\begin{eqnarray*}
\langle\Phi_{p}(\sigma_{p}^{V}(v)),\Phi_{p}(\sigma_{p}^{V}(v^{'}))\rangle=
\langle\Phi(v),\Phi(v^{'})\rangle=\varphi(\langle v,v^{'}\rangle)
&=&\varphi_{p} \circ \pi_{p}(\langle v,v^{'}\rangle)\\
&=&\varphi_{p}(\langle\sigma_{p}^{V}(v),\sigma_{p}^{V}(v^{'})\rangle).
\end{eqnarray*}
Then, by definition of $\Phi_{p}$, the representation $\Phi$ is non-degenerate
if and only if $\Phi_{p}$ is non-degenerate.
\end{proof}
Let $V$ and $W$ be two full Hilbert modules over locally C*-algebras $A$ and $B$,
respectively. Let $E$ be a Hilbert $B$-module, $\Psi:A\to L_{B}(E)$  a
non-degenerate continuous $*$-morphism and $\Phi:W\to B(H,K)$  a non-degenerate
representation of $W$. We construct a non-degenerate representation from $W$ to
$V$ via $E$ as follows.
\begin{construction}\label{const2}
We define a sesquilinear form $\langle.,.\rangle$
on the vector space $E\otimes_{alg}H$ by
$\langle x\otimes h,y\otimes k\rangle=\langle h,\varphi(\langle x,y\rangle)k\rangle_{H}$ and
make the Hilbert space $_{E}H$ as in Construction \ref{const1}.
The map $_{E}^{A}\varphi:A\to B(_{E}H)$ defined by
\[_{E}^{A}\varphi(a)(x\otimes h)=\Psi(a)x\otimes h,~~ a\in A,~ x\in E,~ h\in H,\]
is a representation of $A$. The representation $(_{E}H, \, _{E}^{A}\varphi)$ is
called the {\it Rieffel induced representation} from $B$ to $A$ via $E$, cf.~\cite{mar4}.
Since $A$ acts as an adjointable operator on Hilbert $B$-module $E$,
we can construct interior tensor product $V \otimes_{\Psi} E$
as a Hilbert $B$-module. Hence, we find the Hilbert spaces $_{E}H$ and $_{V\otimes_{\Psi} E}H$.
Let $v\in V$; then the map $E\times H \to \, _{V\otimes_{\Psi} E}H$,
$(x,h) \mapsto v\otimes x\otimes h$ is a bilinear form and so there is a
unique linear transformation $_{E}\Phi(v):E\otimes_{alg}H \to \, _{V\otimes_{\Psi} E}H$
which can be extended to a bounded linear operator
$_{E}^{V}\Phi(v)$ from $_{E}H$ to $_{V\otimes_{\Psi} E}H$. To see this,
suppose $q \in S(B)$, $x\in E$, $h\in H$ and $(\varphi_{q},H)$ is a representation of $B_{q}$
associated to $(\varphi,H)$. We have

\begin{eqnarray*}
\langle \, _{E}\Phi(v)(x\otimes h) &,& _{E}\Phi(v)(x\otimes h) \rangle= \langle v\otimes x\otimes h,v\otimes x\otimes h \rangle
\\&=& \langle h,\varphi(\langle v\otimes x,v\otimes x\rangle)h \rangle_{H}
\\&=&  \langle h,\varphi(\langle x,\Psi(\langle v,v\rangle)x\rangle)h \rangle_{H}
\\&=&  \langle h,\varphi_{q} \circ \pi_{q} (\langle\Psi(\langle v,v\rangle)^{1/2}x,\Psi(\langle v,v\rangle)^{1/2}x\rangle )h \rangle_{H}
\\&=&  \langle h,\varphi_{q} (\langle\sigma_{q}(\Psi(\langle v,v\rangle)^{1/2}x),\sigma_{q}(\Psi(\langle v,v\rangle)^{1/2}x)\rangle )h \rangle_{H}
\\&=&  \langle h,\varphi_{q} ( \langle(\pi_{q})_{*}(\tiny{\Psi}(\langle v,v\rangle)^{1/2})(\sigma_{q}(x)),(\pi_{q})_{*}(\Psi(\langle v,v\rangle)^{1/2})(\sigma_{q}(x)) \rangle )h \rangle_{H}
\\&\leq &\tilde{q}(\Psi\langle v,v\rangle) \langle h,\varphi_{q}(\langle\sigma_{q}(x),\sigma_{q}(x)\rangle)h \rangle_{H}
\\&=& \tilde{q}(\Psi\langle v,v\rangle) \langle h,(\varphi_{q} \circ \pi_{q})(\langle x,x\rangle)h \rangle_{H}
\\&=& \tilde{q}(\Psi\langle v,v\rangle) \langle h,\varphi(\langle x,x\rangle)h \rangle_{H}
\\&=&\tilde{q}(\Psi\langle v,v\rangle) \langle x\otimes h,x\otimes h \rangle.
\end{eqnarray*}
The following equalities hold for every $v,v^{'}\in V,$ $x,x^{'}\in E$ and $h,h^{'}\in H$
\begin{eqnarray*}
\langle x\otimes h\  ,\ _{E}^{V}\Phi^{*}(v)\ _{E}^{V}\Phi(v^{'})
(x^{'}\otimes h^{'}) \rangle &=&\langle_{E}^{V}\Phi(v)(x\otimes h)\ ,\ _{E}^{V}\Phi(v)(x^{'}\otimes h^{'})\rangle
\\&=& \langle v\otimes x \otimes h\ ,\ v^{'}\otimes x^{'}\otimes h^{'}\rangle
\\&=& \langle h,\varphi(\langle v\otimes x,v^{'}\otimes x^{'}\rangle)h\rangle_{H}
\\&=& \langle h,\varphi(\langle x,\Psi(\langle v,v^{'}\rangle)x^{'}\rangle)h^{'}\rangle_{H}
\\&=& \langle x\otimes h\ ,\ \Psi(\langle v,v^{'}\rangle)x^{'}\otimes h^{'}\rangle
\\&=& \langle x\otimes h\ ,\ _{E}^{A}\varphi(\langle v,v^{'}\rangle)(x^{'}\otimes h^{'})\rangle,
\end{eqnarray*}
which imply $\langle_{E}^{V}\Phi(v),_{E}^{V}\Phi(v^{'})\rangle=\,
_{E}^{V}\Phi^{*}(v)\ _{E}^{V}\Phi(v^{'})=\, _{E}^{A}\varphi(\langle v,v^{'}\rangle).$
That is, the map $_{E}^{V}\Phi:V \to B(_{E}H  ,\, _{V\otimes_{\Psi} E}H)$ is a
$_{E}^{A}\varphi$-morphism and so it is a representation of $V$. We now show that
$_{E}^{V}\Phi$ is non-degenerate.
To see this, recall that
$\overline{\Psi(A) (E)}=E$ and $\overline{ \langle V,V \rangle }=A$, which imply
$\overline{\Psi(\langle V,V \rangle ) (E)}=E$. Suppose $x,x^{'}\in E$ and $h\in H$, we have
\begin{eqnarray*}
\|(x-x^{'})\otimes h\|^2 &=&\langle h,\varphi(\langle x-x^{'},x-x^{'}\rangle)h\rangle_{H}
\\&\leq &\| h\|^{2}\|\varphi(\langle x-x^{'},x-x^{'}\rangle)\|
\\&\leq &\| h\|^{2}{q}(\langle x-x^{'},x-x^{'}\rangle) = \| h\|^{2}\bar{q}_{E}(x-x^{'}).
\end{eqnarray*}
Given $\epsilon >0$, there exist $v_{i},v^{'}_{i} \in V$ and $x_{i}\in E$ such that
$\bar{q}_E(\sum_{i}\Psi(\langle v_{i},v^{'}_{i}\rangle)x_{i}-x)<\epsilon$. In view of the
above inequality, the term $\sum_{i}\Psi(\langle v_{i},v^{'}_{i}\rangle)x_{i}\otimes h$
approximates $x\otimes h$ in $_{E}H$. But we have
\begin{eqnarray*}
\sum_{i} \Psi(\langle v_{i},v^{'}_{i} \rangle)x_{i} \otimes h &=&
\sum_{i} \ _{E}^{A}\varphi (\langle v_{i},v^{'}_{i} \rangle)(x_{i} \otimes h)\\
&=&\sum_{i} \ _{E}^{V}\Phi^{*}(v_{i}) \ _{E}^{V}\Phi(v^{'}_{i})(x_{i}\otimes h) \\
&=&\sum_{i} \ _{E}^{V}\Phi^{*}(v_{i})(\ v^{'}_{i}\otimes x_{i}\otimes h),
\end{eqnarray*}
which implies $_{E}^{V}\Phi(V)^{*}(_{V\otimes_{\Psi} E}H)= \, _{E}H$.
The equality $_{E}^{V}\Phi(V)(_{E}H)= \,  _{V\otimes_{\Psi} E}H$ follows from the definition
of $_{E}^{V}\Phi$, i.e.,  $_{E}^{V}\Phi$ is non-degenerate.
\end{construction}

\begin{definition}
The representation $_{E}^{V}\Phi$  in Construction \ref{const2} is a called Rieffel induced representation from
$W$ to $V$ via $E$.
\end{definition}

\begin{theorem}\label{Rep6}
Let $V$ and $W$ be two full Hilbert modules over locally C*-algebras $A$ and $B$,
respectively. Let $E$ be a Hilbert
$B$-module, $\Psi:A\to L_{B}(E)$ a non-degenerate continuous $*$-morphism
and $\Phi:W \to B(H,K)$
 a non-degenerate representation. If $q \in S(B)$ and $(\varphi_{q},H)$
is a non-degenerate representation of $B_{q}$
associated to $(\varphi,H)$, then there is $p\in S(A)$ such
that $A_{p}$ acts non-degenerately on $E_{q}$
and the representations $_{E}^{V}\Phi$ and
$_{E_{q}}^{V_{p}}\Phi_{q} \, \circ \,  \sigma_{p}^{V}$ of $V$ are unitarily equivalent.
\end{theorem}

\begin{proof}
Continuity of $\Psi$ implies that there exists $p\in S(A)$ such that $\tilde{q}(\Psi(a))\leq p(a)$ for each $a\in A$,
which guarantees $\Psi_{p}:A_{p}\to L_{B_{q}}(E_{q})$, $\Psi_{p}(\pi_{p}(a))=(\pi_{q})_{*}(\Psi(a))$
is a $*$-morphism of C*-algebras. Moreover, $\Psi_{p}$ is non-degenerate since
\begin{eqnarray*}
\overline{\Psi_{p}(A_{p})(E_{p})}=\overline{\Psi_{p}(\pi_{p}(A))(\sigma_{p}^{E}(E))}&=&\overline{(\pi_{q})_{*}(\Psi(A)\sigma_{q}^{E}(E))}
\\&=&\sigma_{q}^{E} \, \overline{(\Psi(A)(E))}
\\&=& \sigma_{q}^{E}(E) = E_{q}.
\end{eqnarray*}
If $\Phi_{q}$ is a non-degenerate representation of $W_{q}$ associated
to $\Phi$, then $_{E_{q}}^{V_{p}}\Phi_{q}:V_{p}\to B(_{E_{q}}H\, ,\ _{V_{p}\otimes_{\Psi_{q}} E_{q}}H)$ defined by $_{E_{q}}^{V_{p}}\Phi_{q}(\sigma_{p}^{V}(v))(\sigma_{q}^{E}(x)\otimes h)=\sigma_{p}^{V}(v)\otimes\sigma_{q}^{E}(x)\otimes h$
is a non-degenerate representation of $V_{p}$ which is also a $_{E_{q}}^{A_{p}}\varphi_{q}$-morphism.
Indeed, $_{E_{q}}^{V_{p}}\Phi_{q}$ is the Rieffel induced representation from
$W_q$ to $V_p$ via $E_q$. Hence, $_{E_{q}}^{V_{p}}\Phi_{q}\, \circ \, \sigma_{p}^{V}$
is a non-degenerate representation of
$V$ and it is a $_{E_{q}}^{A_{p}}\varphi_{q}\, \circ \, \pi_{p}$-morphism. The representations
$(_{E}^{A}\varphi\ ,\ _{E}H)$ and $(_{E_{q}}^{A_{p}}\varphi_{q}\  \circ \ \pi_{p}\ ,\ _{E_{q}}H)$ of $A$
are unitarily equivalent by ~\cite[proposition 3.4]{mar4}. We define the linear map
$U_{1}: E\otimes_{alg} H \to E_{q}\otimes_{alg} H$,
$U_{1}(x\otimes h)=\sigma_{q}^{E}(x)\otimes h$ which satisfies
\begin{eqnarray*}
\langle U_{1}(x\otimes h),U_{1}(x\otimes h)\rangle &=&\langle\sigma_{q}^{E}(x)\otimes h,\sigma_{q}^{E}(x)\otimes h\rangle
\\&=& \langle h,\varphi_{q}(\langle\sigma_{q}^{E}(x),\sigma_{q}^{E}(x)\rangle)h\rangle_{H}
\\&=& \langle h,\varphi_{q}(\pi_{q}(\langle x,x\rangle))h\rangle_{H}
\\&=& \langle h,\varphi(\langle x,x\rangle)h\rangle_{H}
\\&=& \langle x \otimes h,x \otimes h\rangle,
\end{eqnarray*}
for all $x\in E$ and $h\in H$. Then $U_{1}$ can be extended to a bounded linear operator,
which is again denoted by $U_{1}$ from $_{E}H$ onto $_{E_{q}}H$.
It is easy to see that $U_{1}$ is a unitary operator.
We define the linear map
$U_{2}: V\otimes_{alg} E\otimes_{alg} H \to V_{p}\otimes_{alg} E_{q}\otimes_{alg} H$ by
$U_{2}(v\otimes x\otimes h)=\sigma_{p}^{V}(v)\otimes \sigma_{q}^{E}(x)\otimes h.$
For every $v\in V$, $x\in E$ and $h\in H$ we have
\begin{eqnarray*}
\langle U_{2}(v\otimes x\otimes h),U_{2}(v\otimes x\otimes h) \rangle
&=& \langle\sigma_{p}^{V}(v)\otimes \sigma_{q}^{E}(x)\otimes h,\sigma_{p}^{V}(v)\otimes \sigma_{q}^{E}(x)\otimes h \rangle
\\  &=&  \langle h,\varphi_{q} \Big(\langle\sigma_{p}^{V}(v)\otimes \sigma_{q}^{E}(x),\sigma_{p}^{V}(v)\otimes \sigma_{q}^{E}(x)\rangle \Big)h \rangle_{H}
\\&=&  \langle h,\varphi_{q}\Big(\langle\sigma_{q}^{E}(x),\Psi_{p}(\langle\sigma_{p}^{V}(v),\sigma_{p}^{V}(v)\rangle)\sigma_{q}^{E}(x)\rangle\Big)h \rangle_{H}
\\&=&   \langle h,\varphi_{q}\Big(\langle\sigma_{q}^{E}(x),\Psi_{p}(\pi_{p}(\langle v,v\rangle))\sigma_{q}^{E}(x)\rangle\Big)h \rangle_{H}
\end{eqnarray*}
\begin{eqnarray*}\langle U_{2}(v\otimes x\otimes h),U_{2}(v\otimes x\otimes h) \rangle&=&  \langle h,\varphi_{q}\Big(\langle\sigma_{q}^{E}(x),(\pi_{q})_{*}(\Psi(\langle v,v\rangle))\sigma_{q}^{E}(x)\rangle\Big)h \rangle_{H}\\
&=&  \langle h,\varphi_{q}\Big(\langle\sigma_{q}^{E}(x),\sigma_{q}^{E}(\Psi(\langle v,v\rangle)x)\rangle\Big)h \rangle_{H}
\\&=&  \langle h,\varphi_{q}\Big(\pi_{q}(\langle x,\Psi(\langle v,v\rangle)x)\Big)h \rangle_{H}
\\&=&  \langle h,\varphi(\langle x,\Psi(\langle v,v\rangle)x)h \rangle_{H}
\\&=& \langle v\otimes x\otimes h,v\otimes x\otimes h \rangle,
\end{eqnarray*}
and so $U_{2}$ can be extended
to a bounded linear operator $U_{2}$ from $_{V\otimes_{\Psi}E}H$ onto $_{V_{p}\otimes_{\Psi_{q}} E_{q}}H$.
It is easy to see that $U_{2}$ is unitary.
Moreover, $U_{2}\  _{E}^{V}\Phi(v)=( _{E_{q}}^{V_{p}}\Phi_{q} \circ \sigma_{p}^{V})\ U_{1}(v)$
for all $v\in V$. Hence, the representations
$_{E}^{V}\Phi$ and $_{E_{q}}^{V_{p}}\Phi_{q} \ \circ \sigma_{p}^{V}$
are unitarily equivalent.
\end{proof}

\begin{theorem}\label{Rep7}
Let $\Phi_{1}:W\to B(H_{1},K_{1})$ and $\Phi_{2}:W\to B(H_{2},K_{2})$ be two
non-degenerate representations of $W$. If $\Phi_{1}$ and $\Phi_{2}$ are
unitarily equivalent, then $_{E}^{V}\Phi_{1}$ and $_{E}^{V}\Phi_{2}$ are unitarily equivalent, too.
\end{theorem}
\begin{proof}
Let $q, q' \in S(B)$, $(\varphi_{1 \, q},H_{1})$ be a representation of $B_{q}$ associated
to $\varphi_{1}$ and let $(\varphi_{2 \, q^{'}}, H_{2})$ be a representation of $B_{q^{'}}$
associated to $\varphi_{2}$. Consider $r\in S(B)$ such that $q,q^{'}\leq r$. By
Theorem \ref{Rep6}, there exists $p\in S(A)$ such that $A_{p}$ acts
non-degenerately on $E_{r}$ and the representation $_{E}^{V}\Phi_{i}$ is
unitarily equivalent to $_{E_{r}}^{V_{p}}\Phi_{i \, r} \ \circ \  \sigma_{p}^{V}$
for $i=1,2$. Since $\Phi_{1 \, r}$ and $\Phi_{2 \, r}$ are
unitarily equivalent representations of $W_{r}$, Lemma \ref{Rep3} implies
that the representations $_{E_{r}}^{V_{p}} \Phi_{1 \, r}$ and
$_{E_{r}}^{V_{p}}\Phi_{2 \, r}$ are unitarily equivalent.
\end{proof}
\begin{corollary}\label{Rep8}
If $\Phi:W\to B(H,K)$  and $\oplus_{i\in I}\Phi_{i}:W\to B(\oplus_{i\in I} H_{i},\oplus_{i\in I}K_{i})$
are unitarily equivalent, then $_{E}^{V}\Phi$ and $\oplus_{i\in I} \ _{E}^{V}\Phi_{i}$ are unitarily equivalent.
\end{corollary}

\begin{proof}
Let $q \in S(B)$ and $\Phi_{q}:W_{q}\to B(H,K)$ be a representation
of $W_{q}$ associated to $\Phi$.
For every $i\in I$, define $\Phi_{i \, q}:W_{q}\to B(H_i,K_i)$ by
$\Phi_{i \, q}(\sigma_{q}^{W}(w))=\Phi_{i}(w).$ If $\sigma_{q}^{W}(w)=0$, then
$\Phi_{q}( \sigma_{q}^{W}( w) )=0$ and so $\Phi(w)=0.$ Since $\Phi$ and $\oplus_{i\in I} \Phi_{i}$
are unitarily equivalent, we conclude that $\oplus_{i \in I}\Phi_{i}(w)=0$
and therefore, $\Phi_{i}(w)=0$ for each $i\in I$. It proves
that $\Phi_{i \, q}$ is well-defined for any $i\in I$. It is easy
to see that $\Phi_{q}$ is unitarily equivalent to $\oplus_{i\in I}\Phi_{i \, q}$.
By Theorem \ref{Rep6}, there exists $p\in S(A)$ such that $A_{p}$
acts non-degenerately on $E_{q}$ and the representations $_{E}^{V}\Phi$
and $_{E_{q}}^{V_{p}}\Phi_{q} \ \circ \  \sigma_{p}^{V}$ of $V$ are
unitarily equivalent. The representations $_{E}^{V}\Phi_{i}$
and $_{E_{q}}^{V_{p}}\Phi_{i \, q} \ \circ \  \sigma_{p}^{V}$, $i\in I$
are unitarily equivalent, too.
On the other hand, Corollary \ref{Rep4} implies that the
representations $_{E_{q}}^{V_{p}}\Phi_{q}$ and
$\oplus_{i\in I} \ _{E_{q}}^{V_{p}}\Phi_{i \, q}$ of $V_{p}$ are unitarily
equivalent. Consequently, the representations
$_{E_{q}}^{V_{p}}\Phi_{q} \ \circ \  \sigma_{p}^{V}$ and
$\oplus_{i\in I}( _{E_{q}}^{V_{p}}\Phi_{i \, q} \ \circ \  \sigma_{p}^{V})$ of $V$
are unitarily equivalent.
\end{proof}

\section{The imprimitivity theorem for Hilbert modules}
In this section, we introduce the concept of Morita equivalence
between Hilbert modules over locally C*-algebras
and give a module version of the imprimitivity theorem.

Let $A$ and $B$ be locally C*-algebras. We say that $A$ and $B$ are {\it strongly
Morita equivalent}, written $A\sim_{M} B$, if there is a full Hilbert $A$ module $E$
such that locally C*-algebras $B$ and $K_{A}(E)$ are isomorphic. Joita
\cite[Proposition 4.4]{mar5} showed that
strong Morita equivalence is an equivalence relation in the set of all
locally C*-algebras. The vector space
$\tilde{E}:=K_{A}(E,A)$ is a full Hilbert $K_{A}(E)$-module
with the following action and inner product
\begin{eqnarray*}
(T,S) & \to  TS,& ~S\in K_{A}(E), ~T\in K_{A}(E,A),\\
\langle T,S\rangle & = T^{*}S,& ~T,S \in K_{A}(E,A).
\end{eqnarray*}
Since locally C*-algebras $B$ and $K_{A}(E)$ are
isomorphic, $\tilde{E}$ may be regarded as a Hilbert $B$-module.
Moreover, the linear map $\alpha$ from $A$ to $K_{B}(\tilde{E})$
defined by $\alpha(a)(\theta_{b,x})=\theta_{ab,x}$ is an
isomorphism of locally C*-algebras by \cite[Lemma 4.2~and~Remark 4.3]{mar5}.
It is easy to see that for each $p\in S(A)$, the linear
map $U_{p}:(\tilde{E})_{p}\to \tilde{E_{p}}$ defined by $U_{p}(T+N_{p}^{\tilde{E}})=(\pi_{p})_{*}(T)$
is unitary and so the Hilbert $K_{A_{p}}(E_{p})$-modules
$(\tilde{E})_{p}$ and $\tilde{E_{p}}$ are the same.

\begin{definition}
Suppose $V$ and $W$ are Hilbert modules over locally C*-algebras $A$ and $B$, respectively.
The Hilbert modules $V$ and $W$ are called  Morita equivalent if $K_{A}(V)$ and $K_{B}(W)$
are strong Morita equivalent as locally C*-algebras. In this case, we write $V\sim_{M}W$.
\end{definition}

\begin{lemma}\label{Rep9}
Let $V$ be a full Hilbert module over locally C*-algebra $A$.
Then $K_{A}(V)$ is strong Morita equivalent to $\overline{ \langle V,V \rangle }$.
\end{lemma}

\begin{proof}
The module $\tilde{V} = K_{A}(V,A)$ is a full Hilbert $K_{A}(V)$-module by \cite[Corollary 3.3]{mar5}.
Then locally C*-algebras $K_{K_{A}(V)}(\tilde{V})$
and $K_{A}(A)$ are isomorphic by Lemma 4.2 in  \cite{mar5}.
Since $\overline{ \langle V,V \rangle }=A \simeq K_{A}(A)$, locally C*-algebras $K_{A}(V)$ and
$\overline{ \langle V,V \rangle }$ are strong Morita equivalent.
\end{proof}
\begin{corollary}\label{Rep10}
Two full Hilbert modules over locally C*-algebras are Morita
equivalent if and only if their underlying locally C*-algebras
are strong Morita equivalent.
\end{corollary}

\begin{theorem}\label{Rep11}
Let $V$ and $W$ be two full Hilbert modules over locally C*-algebras $A$ and $B$,
respectively, such that $V\sim_{M}W$. If $E$ is a Hilbert $A$-module which gives
the strong Morita equivalence between $A$ and $B$, and $\Phi$ is a non-degenerate representation
of $V$, then $\Phi$ is unitarily equivalent to $_{\tilde{E}}^{V}(_{E}^{W}\Phi)$.
\end{theorem}

\begin{proof}
Let $p\in S(A)$ and $\Phi_{p}$ be a non-degenerate representation of $V_{p}$ associated to $\Phi$.
Using \cite[Lemma 4.1]{mar4}, there is $q\in S(B)$ such that $A_{p}\sim_{M}B_{q}$ and
$E_{p}$ gives the strong Morita equivalent between $A_{p}$ and $B_{q}$.
The representations $\varphi_{p}$
and $_{\tilde{E}_{p}}^{A_{p}}(_{E_{p}}^{B_{q}}\varphi_{p})$ of $A_{p}$
are unitarily equivalent by \cite[Theorem 3.29]{rae}. Then the representations $\Phi_{p}$
and $_{\tilde{E}_{p}}^{V_{p}}(_{E_{p}}^{W_{q}}\Phi_{p})$ of $V_{p}$
are unitarily equivalent by Lemma \ref{Rep2} and consequently, the
representations $_{\tilde{E}_{p}}^{V_{p}}(_{E_{p}}^{W_{q}}\Phi_{p})\ \circ \ \sigma_{p}^{V}$
and $\Phi_{p}\ \circ \ \sigma_{p}^{V}=\Phi$ of $V$ are unitarily equivalent.
In view of Theorems \ref{Rep6} and \ref{Rep7}, we have
\begin{itemize}
  \item  the representations
$_{E}^{W}\Phi$ and $_{E_{p}}^{W_{q}}\Phi_{p}\ \circ \ \sigma_{q}^{W}$ of $W$
are unitarily equivalent,
  \item the representations $_{\tilde{E}}^{V}(_{E}^{W}\Phi)$
and $_{\tilde{E}}^{V}(_{E_{p}}^{W_{q}}\Phi_{p}\ \circ \ \sigma_{q}^{W})$ of $V$ are unitarily equivalent, and
  \item the representations $_{\tilde{E}}^{V}(_{E_{p}}^{W_{q}}\Phi_{p}\ \circ \ \sigma_{q}^{W})$
and $ _{\tilde{E_{P}}}^{V_{p}}( _{E_{p}}^{W_{q}}\Phi_{p}\ \circ \ \sigma_{q}^{W})_{q}\ \circ \ \sigma_{p}^{V}$
of $V$ are unitarily equivalent.
\end{itemize}
The assertion now follows from the fact that
$( _{E_{p}}^{W_{q}}\Phi_{p}\ \circ \ \sigma_{q}^{W})_{q}=\ _{E_{p}}^{W_{q}}\Phi_{p}$.
\end{proof}
We now reformulate the imprimitivity theorem within the framework of Hilbert modules as follows.
\begin{theorem}\label{Rep12}
Let $V$ and $W$ be two Hilbert modules over locally
C*-algebras $A$ and $B$, respectively. If $V\sim_{M}W$, then there is a bijective
correspondence between equivalence classes of non-degenerate representations of $V$ and $W$.
\end{theorem}
\begin{proof}
By replacing the underlying C*-algebras $A$ and $B$, we may
assume that $V$ and $W$ are full Hilbert modules over $A$ and $B$, respectively.
Let $E$ be a Hilbert $A$-module which gives  strong Morita equivalence
between $A$ and $B$. Then, by Theorems \ref{Rep7} and \ref{Rep11},
the map $\Phi \mapsto \, _{E}^{W}\Phi$
from the set of all non-degenerate representations of $V$  to
the set of all non-degenerate representations of $W$ induces a bijective
correspondence between equivalence classes of non-degenerate representations of $V$ and $W$.
\end{proof}

{\bf Acknowledgement}: The authors would like to thank
the referee for his/her careful reading and useful comments.

\end{document}